\documentclass[draft]{amsart}
\usepackage{amssymb}
\usepackage{amsfonts}
\usepackage{latexsym}
\usepackage{euscript}
\usepackage{mathrsfs} 
\usepackage{units}   
\usepackage{color}

\makeatletter
\@namedef{subjclassname@2010}{%
\textup{2010} Mathematics Subject Classification}
\makeatother

\numberwithin{equation}{section}
\newtheorem{thm}{Theorem}[section]
\newtheorem{cor}[thm]{Corollary}
\newtheorem{lem}[thm]{Lemma}
\newtheorem{pro}[thm]{Proposition}

\newtheorem*{thm*}{Theorem}

\theoremstyle{remark}
\newtheorem{rem}[thm]{Remark}

\newtheorem*{opq}{Question}

\theoremstyle{definition}

\def\N{\mathbb{N}}
\def\oo{\mathcal O}

\DeclareMathOperator{\D}{d\!}
\DeclareMathOperator{\lin}{lin}
\def\ee{\mathcal{E}}
\def\ff{\mathcal{F}}
\def\hh{\mathcal H}
\newcommand\miu[1]{\mu\circ\phi^{-1}(#1)}

\newcommand*{\cbb}{\mathbb C}
\newcommand*{\dz}[1]{{\EuScript D}(#1)}
\newcommand*{\Ge}{\geqslant}
\newcommand*{\hsf}{\mathsf h}
\newcommand*{\is}[2]{\langle#1,#2\rangle}
\newcommand*{\Le}{\leqslant}

\newcommand*{\zbb}{\mathbb Z}

\begin{document}
\title[On unbounded composition operators in $\ell^2$-spaces]{On normal, formally normal and quasinormal composition operators in $\ell^2$-spaces}
\author[P.\ Budzy\'{n}ski]{Piotr Budzy\'{n}ski}
\address{Katedra Zastosowa\'{n} Matematyki, Uniwersytet Rolniczy w Krakowie, ul.\ Balicka 253c, 30-198 Krak\'ow, Poland}
\email{piotr.budzynski@ur.krakow.pl}
\thanks{The research was partially conducted when the author served as a Research Fellow at the Mathematical Institute of the Polish Academy of Sciences}

\begin{abstract}
Unbounded composition operators in $L^2$-space over discrete measure spaces are investigated.
Normal, formally normal and quasinormal composition operators acting in $L^2$-spaces of this kind are characterized.
\end{abstract}

\maketitle
\section{Introduction}
Composition operator on $L^2$-spaces, which has been introduced during the process of formalization of classical mechanics, appear in many areas of mathematics. Theory of these operators turned out to be fruitful and is well-developed now. Many properties of bounded composition operators obtained elegant measure theoretic characterizations (see a monograph \cite{sin-man} and references therein). The motivation for a research on unbounded composition operators is two-fold. Firstly, this class of operators is known to contain examples which have very interesting properties, sometimes quite surprising (cf.\ \cite{jab-jun-sto-Triv, jab-jun-sto-Stiel, bud, bud-jab-jun-sto-Comp}). Secondly, the theory of these operators is relatively young and, although it draw more attention recently (cf.\ \cite{cam-hor, jab, bud-jab-jun-sto-Comp, bud-jab-jun-sto-Subn, bud-dym-jab-sto}), there is a plenty of room for development.

In the following note we address the question of normality, formal normality and quasinormality of unbounded composition operators acting in an $L^2$-spaces over discrete measure spaces (in further parts of the paper we refer to such spaces as $\ell^2$-spaces). These issues were considered in a general $L^2$-space setting in \cite{bud-jab-jun-sto-Comp}. Naturally, some information is lost then. We fill up this gap to some extent in the present paper. The $\ell^2$-space setting is very convenient when approaching many problems concerning unbounded composition operators in $L^2$-spaces because of two reasons.  The first is that in this setting the proofs are rather elementary and do not require advanced measure theoretic methods. The second is that the results give an insight into what happens in a general setting (from this point of view the present paper can be seen as an addendum to \cite{bud-jab-jun-sto-Comp}).
\section{Prerequisites}
In all what follows $\zbb$, $\zbb_+$ and $\cbb$ stand for the sets of integers, nonnegative integers and complex numbers, respectively. Given a set $X$, by $\ee=\ee_X$ we denote the family $\{\chi_{\{x\}}\colon x\in X\}$; as usual, $\chi_\omega$ denotes the characteristic function of the a set $\omega$. If $\ff$ is a subset of a Hilbert space, then $\lin \ff$ denotes the linear span of $\ff$.

Let $\hh$ be a (complex) Hilbert space and let $A$ be a (linear) operator $A$ acting in $\hh$. Then $\dz{A}$ denotes the domain of $A$ and $A^*$ stands for the adjoint $A$ (if it exists).
If $A$ is closed and densely defined, then we call $A$ {\em normal} if and only if $A^*A=AA^*$. If $A$ is densely defined, $\dz{A}\subseteq \dz{A^*}$ and $\|A f\|= \|A^* f\|$ for every $f\in\dz{A}$, then $A$ is {\em formally normal}. $A$ is said to be {\em quasinormal} if $A$ is densely defined and $U|A|\subseteq|A|U$, where $A=U|A|$ is its polar decomposition.

Let $(X,\varSigma,\mu)$ be a $\sigma$-finite measure space and let $\phi\colon X\to X$ be a $\varSigma$-measurable mapping. Define
\begin{align*}
    \mu\circ\phi^{-1}(\varDelta)=\mu\Big(\phi^{-1}(\varDelta)\Big),\quad \varDelta\in \varSigma.
    \end{align*}
Suppose that $\phi$ is {\em nonsingular}, i.e., the measure $\mu\circ \phi^{-1}$ is absolutely continuous with respect to $\mu$. Then the so-called {\em composition operator}
    $$C_\phi\colon L^2(\mu) \supseteq\dz{C_\phi} \to L^2(\mu),$$
given by the formula
    \begin{align*}
    \dz{C_\phi}=\Big\{f \in L^2(\mu) \colon f\circ \phi \in L^2(\mu)\Big\},\quad C_\phi f=f\circ\phi,\quad f\in \dz{C_\phi},
   \end{align*}
is well-defined linear operator in $L^2(\mu)$ (as usual, $L^2(\mu)$ stands for the Hilbert space of all complex-valued and square-integrable with respect to $\mu$ functions on $X$). The reverse is also true, i.e., if $\phi\colon X\to X$ is $\varSigma$-measurable mapping such that the map $f\mapsto f\circ \phi$ defines a linear operator in $L^2(\mu)$, then $\phi$ must be nonsingular. The operator $C_\phi$ is automatically closed.
Properties of $C_\phi$ are usually written in terms of the canonical Radon-Nikodym derivative $\hsf_\phi$ associated to $C_\phi$ and given by
    \begin{align*}
    \hsf_\phi=\frac{\D\mu\circ\phi^{-1}}{\D\mu}.
    \end{align*}
In particular, $C_\phi$ is bounded if and only if $\hsf_\phi$ is essentially bounded. For more information about bounded and unbounded composition operators the reader is referred to \cite{sin-man,nor} and \cite{bud-jab-jun-sto-Comp}.

In this paper we deal with composition operators acting in $L^2$-spaces over a discrete measure spaces. To establish the notation and terminology let us consider a countable set $X$.
By $2^X$ we denote the power set of $X$.  Let $\mu$ be a (positive) $\sigma$-finite measure on $2^X$ (note that $\sigma$-finiteness of $\mu$ is equivalent to  $\mu(x)<\infty$ for all $x\in X$; for brevity, we write $\mu(x)$ instead a more formal expression $\mu(\{x\})$). In this context we write
$\ell^2(\mu)$ instead to $L^2(\mu)$ and when mentioning $\ell^2$-space we mean $\ell^2(\mu)$ for some measure $\mu$ on $2^X$ with countable $X$.
Let $\phi\colon X\to X$. Clearly, $\phi$ is $2^X$-measurable. Nonsingularity of $\phi$ is equivalent to the following condition
\begin{align}\label{18.04.11.01}
\forall x\in X\quad \mu(x)=0\quad\Longrightarrow\quad\miu{x}=0.
\end{align}
Obviously, if $\mu(x)>0$ for every $x\in X$, then $\phi$ is automatically nonsingular.
If  $\phi$ is nonsingular, then we will write $[\phi]$ for the set of all mappings $\psi\colon X\to X$ such that
$\{x\in X\colon \phi(x)\neq\psi(x)\}\subseteq X\setminus  X_{+}$, where $ X_{+}=\{x\in X\colon \mu(x)\neq0\}$.
Any such a mapping we call a {\em representative} of $[\phi]$. As easily seen, every $\psi\in[\phi]$ is nonsingular and $C_\phi=C_\psi$.
Now, assume that $\phi$ is nonsingular. Set $\hsf_\phi(x)=0$ for $x\in X\setminus X_{+}$. Then we have
    \begin{align}\label{rnatomic}
    \hsf_\phi(x)=\frac{\miu{x}}{\mu(x)},\quad x\in X,
    \end{align}
and\footnote{\;Here, and later, we adhere to the convention that $\sum_{x\in\varnothing}\alpha(x)=0$.}
    \begin{align}\label{17.01.11.01}
    \dz{C_\phi}=\big\{f \in \ell^2(\mu) \colon \sum_{x\in X_{+}}|f(x)|^2\mu\circ\phi^{-1}(x)<\infty\big\}.
   \end{align}
In our context it seems natural to avoid using Radon-Nikodym derivative $\hsf_\phi$, at least explicitly, and express properties of the composition operator $C_\phi$ directly in terms of the measure $\mu$, the transformation $\phi$ and the transport measure $\mu\circ\phi^{-1}$.

Suppose that $\phi\colon X\to X$ is a bijective mapping. For $x\in X$, we define
    \begin{align*}
    \oo_\phi(x)=\{\phi^n(x)\colon n\in \zbb_+\}\cup\{(\phi^n)^{-1}(x)\colon n\in\zbb_+\},
    \end{align*}
where $\phi^k$ denotes $k$-fold composition of transformation with itself and $\phi^0=\mathrm{id}_X$ (the identity mapping on $X$).
We call $\oo_\phi(x)$ {\em the orbit of $\phi$ at $x$}. If no confusion can arise we write $\oo(x)$ instead of $\oo_\phi(x)$.
It is easy to see that the orbits $\oo(x)$ and $\oo(y)$ either have empty intersection or are equal.
Observe also that if $y\in\oo(x)$, then $\oo(x)=\oo(y)$.

Below we gather some useful information about $C_\phi$, $C_\phi^*$ and their products. They can be deduced from \cite{cam-hor, bud-jab-jun-sto-Comp, bud}. We include elementary proofs for completeness and the reader's convenience. Recall that $\ee=\{\chi_{\{x\}}\colon x\in X\}$.
\begin{lem}\label{01.21.04.09}
Let $X$ be a countable set, $\mu$ be a $\sigma$-finite measure on $2^X$ and $\phi\colon X\to X$ be nonsingular. Then the following conditions hold.
    \begin{enumerate}
    \item $\tau(X_+)\subseteq X_+$ for every representative $\tau$ of $[\phi]$ and there exists a representative $\psi$ of $[\phi]$ such that $\psi(x)=x$ for every $x\in X\setminus X_+$.
    \item $\ee\subseteq\dz{C_\phi}$ if and only if the following condition holds$:$
\begin{align}\label{star}
\miu{x}<\infty, \quad x\in  X_{+}.
\end{align}
    \item $C_\phi$ is densely defined if and only if $\ee\subseteq\dz{C_\phi}$.
    \item If \eqref{star} holds, then $\ee\subseteq\dz{C_\phi^*}$.
    \item If \eqref{star} holds, $z\in X$ and $\varDelta\subseteq\phi^{-1}(\{z\})$, then $\chi_\varDelta\in\dz{C_\phi^*}$.
    \item If \eqref{star} holds and $f\in\dz{C_\phi^*}$, then
        \begin{align}\label{02.21.04.09}
            (C_\phi^*f)(x)=
                \sum_{z\in\phi^{-1}(\{x\})}\frac{f(z)\mu(z)}{\mu(x)},\quad x\in  X_{+}.
        \end{align}
    \item If \eqref{star} holds, then $\ee\subseteq\dz{C_\phi^*C_\phi}$ and
        \begin{align}\label{03.21.04.09}
        C_\phi^* C_\phi \chi_{\{x\}}=\frac{\miu{x}}{\mu(x)}\cdot\chi_{\{x\}},\quad x\in  X_{+}.
        \end{align}
    \item If \eqref{star} holds, then $\ee\subseteq\dz{C_\phi C_\phi^*}$ and
        \begin{align}\label{04.21.04.09}
        C_\phi C_\phi^* \chi_{\{x\}}=\frac{\mu(x)}{\mu(\phi(x))}\cdot \chi_{\{\phi(x)\}}\circ \phi,\quad x\in  X_{+} 
        \end{align}
    \end{enumerate}
\end{lem}
\begin{proof}
Suppose that $\tau\colon X\to X$ is a mapping such that $C_\tau$ is well defined and $C_\tau=C_\phi$. Then $\tau$ need to satisfy \eqref{18.04.11.01} and this implies that $\tau( X_{+})\subseteq X_{+}$. Values of $\tau$ over the set $X\setminus X_{+}$ do not affect values of $C_\psi$
and hence we may modify $\tau$ on $X\setminus X_{+}$ so as to have $\tau(x)=x$ for $x\in X\setminus X_{+}$.
Clearly, the modified mapping, denote it by $\psi$, is a representative of $\phi$ and satisfies (1).

Let $x\in X_+$. Then, by \eqref{17.01.11.01}, $\chi_{\{x\}}\in\dz{C_\phi}$ if and only if $\miu{x}<\infty$. Evidently, this implies (2).

Clearly, if $\ee\subseteq \dz{C_\phi}$, then $C_\phi$ is densely defined. Now, suppose that $C_\phi$ is densely defined.
For $f\in\dz{C_\phi}$ we set $X_f=\{x\in X_+\colon f(x)\neq 0\}$. Observe that if $x_0\in X_f$, then, by \eqref{17.01.11.01},  $\chi_{\{x_0\}}\in\dz{C_\phi}$ because we have
    \begin{align*}
    |f(x_0)|^2\cdot\sum_{x\in X} |\chi_{\{x_0\}}(x)|^2\miu{x}\Le\sum_{x\in X} |f(x)|^2\miu{x}<\infty.
    \end{align*}
Clearly, $X=\big( X\setminus X_+\big) \cup \bigcup_{f\in\dz{C_\phi}}X_{f}$. Thus $\ee$ is contained in $\dz{C_\phi}$.

By (2) and (3), $C_\phi$ is densely defined and hence $C_\phi^*$ is well-defined. Since
    \begin{align*}
    \big|\langle \chi_{\{x\}}, C_\phi f\rangle\big|
    &=
    \Big|\sum_{z\in X} \chi_{\{x\}}(z)\overline{ f(\phi(z))}\mu(z)\Big|
    =
    \big|f({\phi(x)})\big|\,\mu(x)\\
    &\Le
    \Big(\sum_{z\in X} |f(z)|^2\mu(z)\Big)^{\frac12} \frac{\mu(x)}{\sqrt{\mu(\phi(x))}},\quad x\in X_+,
    \end{align*}
we get that $\ee\subseteq\dz{C_\phi^*}$.

Let $z\in X_+$ and $\varDelta\subseteq\phi^{-1}(\{z\})$. Using \eqref{star} and reasoning as in the proof of (4), i.e.,
    \begin{align*}
    \big|\langle \chi_{\varDelta},C_\phi f\rangle\big|
    \Le
    \sum_{y\in\phi^{-1}(\{z\})}|f(z)|\,\mu(y)
    \Le
    \|f\|\cdot \miu{z},
    \end{align*}
we easily get (5).

For the proof of (6) observe that
    \begin{align*}
    (C_\phi^*f)(x)\, \mu(x)=\langle C_\phi^* f, \chi_{\{x\}}\rangle=\langle f, C_\phi \chi_{\{x\}}\rangle= \sum_{z\in X} f(z) \chi_{\phi^{-1}(\{x\})}(z) \mu(z).
    \end{align*}
for every $f\in\dz{C_\phi^*}$ and $x\in X_+$.

Let $x\in X_+$. By (2), (3) and (5) we know that $\chi_{\{x\}}\in\dz{C_\phi}$ and $\chi_{\phi^{-1}(\{x\})}\in\dz{C_\phi^*}$. This and $C_\phi \chi_{\{x\}}=\chi_{\phi^{-1}(\{x\})}$  implies that
$\chi_{\{x\}}\in\dz{C^*_\phi C_\phi}$. Moreover, for every $y\in X_+$ we have
    \begin{align*}
    \big(C_\phi^*C_\phi \chi_{\{x\}}\big)(y)
    &=
    \big(C_\phi^* \chi_{\phi^{-1}(\{x\})}\big)(y)
    =
    \sum_{z\in\phi^{-1}(\{y\})}\frac{\chi_{\phi^{-1}(\{z\})}(w)\mu(z)}{\mu(y)}
    \end{align*}
which yields \eqref{03.21.04.09}.

Let $x\in X_+$. By (4) we have $\chi_{\{x\}}\in\dz{C_\phi^*}$. Hence \eqref{02.21.04.09} implies
    \begin{align*}
    (C_\phi^*\chi_{\{x\}})(y)
    =
    \sum\limits_{z\in\phi^{-1}(\{y\})}\frac{\chi_{\{x\}}(z)\mu(z)}{\mu(y)}
    =\frac{\mu(x)}{\mu(\phi(x))}\cdot\chi_{\phi(x)} (y),\quad y\in X_+.
    \end{align*}
This, (2)  and (3) imply that $C_\phi^*\chi_{\{x\}}\in\dz{C_\phi}$. Moreover
    \begin{align*}
    (C_\phi C_\phi^*\chi_{\{x\}})(y)
    =
    (C_\phi^*\chi_{\{x\}})({\phi(y)})
    =\frac{\mu(x)}{\mu(\phi(x))}\cdot \chi_{\phi(x)}\big({\phi(y)}\big), \quad y\in X_+.
    \end{align*}
This gives \eqref{04.21.04.09}.
\end{proof}
\begin{rem}
By Lemma \ref{01.21.04.09}, the domain of $C_\phi$ is dense in $\ell^2(\mu)$ whenever condition \eqref{star} is satisfied.
It is worth mentioning that condition \eqref{star} doesn't imply density of the domain of $C_\phi^2$.
For example there exists a hyponormal composition operator such that $\dz{C_\phi^2}=\{0\}$ (cf. \cite{bud, jab-jun-sto-Triv}; see also \cite{bud-dym-jab-sto}).
\end{rem}
{\bf Caution}: Throughout the rest of the paper we shall assume that $X$ is a countable set, $\mu$ is a $\sigma$-finite measure on $2^X$ and $\phi\colon X\to X$ is nonsingular.
\section{Normality classes}
By Lemma \ref{01.21.04.09}, if condition \eqref{star} is satisfied, then $\ee\subseteq \dz{C_\phi C_\phi^*}\cap\dz{C_\phi^*C_\phi}$. Below we show necessary conditions for the equality $C_\phi C_\phi^*|_{\lin \ee}=C_\phi^* C_\phi|_{\lin\ee}$ to hold. Later, using this information, we will characterize normal composition operators.
\begin{lem}\label{01.08.06.09}
Suppose that condition \eqref{star} is satisfied. If
    \begin{align}\label{01.04.06.09}
        C_\phi C_\phi^* f=C_\phi^* C_\phi f,\quad f\in\ee,
    \end{align}
then there exists a representative $\psi$ of $[\phi]$ such that
    \begin{enumerate}
    \item[(ii)] $\psi$ is bijective and $\psi(x)=x$ for every $x\in X\setminus X_+$,
    \item[(ii)] for all $x\in X_+$ and $k\in\zbb_+$,
                $\frac{\mu(x)}{\mu(\psi(x))}=\frac{\mu(\psi^k(x))}{\mu(\psi^{k+1}(x))}$.
    \end{enumerate}
\end{lem}
\begin{proof}
Let $\psi$ be the representive of $[\phi]$ given by Lemma \label{01.21.04.09}\,(1). Suppose that $\psi$ is not bijective. If $\psi$ is not injective, then there exists $x\in  X_{+}$ such that the set $\psi^{-1}(\{\psi(x)\})$ contains at least two distinct points $x,x'\in X_{+}$.
By \eqref{04.21.04.09}, we get
\begin{align*}
(C_\psi C_\psi^*\chi_{\{x\}})(x)=(C_\psi C_\psi^*\chi_{\{x\}})(x')=\frac{\mu(x)}{\mu(\psi(x))}
\neq 0.
\end{align*}
This contradicts \eqref{01.04.06.09} since, by \eqref{03.21.04.09}, we have
$(C_\psi^* C_\psi \chi_{\{x\}})(x')=0$. Hence $\psi$ must be injective. Suppose now that $X\setminus\psi( X)\neq\varnothing$. Choose
$x\in \big(X\setminus\psi(X)\big)\cap X_{+}$ and observe that by \eqref{03.21.04.09} we have $(C_\psi^* C_\psi \chi_{\{x\}})(x)=0$. By \eqref{04.21.04.09} we get $(C_\psi C_\psi^* \chi_{\{x\}})(x)=\frac{\mu(x)}{\mu(\psi(x))}\neq0$, a contradiction with \eqref{01.04.06.09}.
This imply that $\psi(X)$ covers the whole of $X$. Consequently (ii) is valid.

Now,  by (i), for every $x\in X$ we have $\mu\circ\psi^{-1}(x)=\mu(\psi^{-1}(x))$. Comparing \eqref{03.21.04.09} and
\eqref{04.21.04.09} and using \eqref{01.04.06.09} we get
    \begin{align*}
    \frac{\mu(\psi^{-1}(x))}{\mu(x)}=\frac{\mu(x)}{\mu(\psi(x))},\quad x\in X_+
    \end{align*}
Reasoning by induction we can prove (ii).
\end{proof}
The next result contains another portion of information enlightening the structure of normal composition operators on $\ell^2$-spaces.
\begin{pro}\label{orbity}
Suppose that condition \eqref{star} and \eqref{01.04.06.09} hold. Then there exists a representative $\psi$ of $[\phi]$ such that$:$
    \begin{enumerate}
    \item[(i)] if $x\in X_+$ and $\oo_\psi(x)$ is finite, then there exists a positive real number $c$ such that $\mu(y)=c$ for all $y\in\oo_\psi(x)$;
    \item[(ii)] if $C_\phi$ is not bounded operator in $\ell^2(\mu)$, then there exists $Y\subseteq X_+$ such that
        \begin{itemize}
            \item $Y$ is infinite,
            \item for every $y\in Y$, the orbit $\oo_\psi(x)$ is infinite,
            \item $\oo_\psi(y)\cap\oo_\psi(y')=\varnothing$ for all $y,y\in Y$ such that $y\neq y'$.
            \end{itemize}
    \end{enumerate}
\end{pro}
\begin{proof}
Let $\psi$ be the representative of $[\phi]$ given by Lemma \ref{01.08.06.09}. We may assume that $\psi(x)=x$ for all $x\in X\setminus X_{+}$ by Lemma \label{01.21.04.09}.

Suppose that $x\in X_+$. Then, by Lemma \ref{01.08.06.09}\,(ii), the sequence $\{\mu(\psi^n(x))\}_{n=0}^\infty$ is a geometric sequence with a common ratio $q>0$.
If $\oo_\psi(x)$ is finite and $\oo_\psi(x)\neq\{x\}$, then there is $n\in\N$ such that $\psi(\psi^{n}(x))=x$.
Hence $\mu(x)=\mu(\psi^{n}(x))=q^{n+1}\mu(x)$, which is possible only in the case $q=1$. All of this gives (i).

For the proof of (ii) suppose that for all $x\in X_+$, the set $\oo_\psi(x)$ is finite. Then by (i), $\frac{\mu\circ\psi^{-1}(x)}{\mu(x)}=1$ for every $x\in X_+$. This and \eqref{rnatomic} implies that $\hsf_\psi=1$ a.e. $[\mu]$. As a consequence $C_\phi=C_\psi$ is bounded, a contradiction.
Therefore, $\oo_\psi(x)$ must be infinite for some $x\in X_+$. The number of disjoint and infinite orbits is necessarily infinite as well.
Indeed, suppose that $\oo_\psi(x_1),\ldots, \oo_\psi(x_k)$ are all pairwise disjoint and infinite orbits of $\psi$, with some $x_1,\ldots x_k\in X$ and $k\in\N$. As we already know, $\{\mu(\psi^n(x_i))\}_{n=0}^\infty$ is a geometric sequence with a common ratio $q_i\Ge 0$. Again by (i) and \eqref{rnatomic}, $\hsf_\psi\Le 1+\max\{q_i\colon i=1,\ldots, k\}$ a.e. $[\mu]$. This leads to a contradiction again. Therefore (ii) holds and the proof is complete.
\end{proof}
\begin{rem}
Regarding assertion (i) of Proposition \ref{orbity} we note that $\oo_\psi(x)=\{x\}$ for every $x\in X\setminus X_+$ and thus $\mu\big(\oo_\psi(x)\big)=0$.
\end{rem}
The characterization of normal composition operators in $\ell^2(\mu)$ reads as follows (cf. \cite[Proposition 8.2]{bud-jab-jun-sto-Comp}).
\begin{thm}\label{21.01.11.01}
The following conditions are equivalent:
    \begin{enumerate}
    \item[(i)] $C_\phi$ is normal,
    \item[(ii)] $\ee\subseteq \dz{C_\phi C_\phi^*}\cap\dz{C_\phi^* C_\phi}$ and $C_\phi C_\phi^*f=C_\phi^* C_\phi f$ for every $f\in\ee$,
    \item[(iii)] there exists a representative $\psi$ of $[\phi]$, which is bijective and satisfy
    \begin{align}\label{21.01.11.03}
    \frac{\mu(x)}{\mu(\psi(x))}=\frac{\mu(\psi^k(x))}{\mu(\psi^{k+1}(x))}, \quad k\in\N, x\in X_+.
    \end{align}
    \end{enumerate}
\end{thm}
\begin{proof}
(i)$\Rightarrow$(ii)\ Evident, since by Lemma \ref{01.21.04.09} we have $\ee\subseteq\dz{C_\phi C_\phi^*}\cap\dz{C_\phi^* C_\phi}$.

(ii)$\Rightarrow$(iii)\ By Lemma \ref{01.21.04.09} condition \eqref{star} is satisfied. The rest follows immediately from Lemma \ref{01.08.06.09}.

(iii)$\Rightarrow$(i) By bijectivity of $\psi$ and \eqref{02.21.04.09} we have
    \begin{align}\label{21.01.11.02}
    (C_\psi^*f)(x)= f(\psi^{-1}(x))\, \frac{\mu(\psi^{-1}(x))}{\mu(x)}, \quad x\in X_+,
    \end{align}
for any $f\in\dz{C_\psi^*}$. Observe that $\dz{C_\psi^*}=\dz{C_\psi}$. Indeed, if $f\in\dz{C_\psi}$, then by \eqref{21.01.11.03} and \eqref{17.01.11.01} we get
\begin{align*}
\sum_{x\in X_+} |f(\psi^{-1}(x))|^2\Big(\frac{\mu(\psi^{-1}(x)}{\mu(x)}\Big)^2\mu(x)
&=\sum_{x\in X}^\infty |f(\psi^{-1}(x))|^2\mu(\psi^{-2}(x))\\
&=\sum_{x\in X}^\infty |f(\psi(x))|^2\mu(x)<\infty
\end{align*}
This, the Cauchy-Schwarz inequality and bijectivity of $\psi$ imply that\allowdisplaybreaks
\begin{align*}
\big|\is{C_\psi g}{f}\big|
&=
\Big|\sum_{x\in X}^\infty g(\psi(x))\overline{f(x)}\mu(x)\Big|\\
&=
\Big|\sum_{x\in X}^\infty g(x)\overline{f(\psi^{-1}(x))}\mu(\psi^{-1}(x))\Big|\\
&=
\Big|\sum_{x\in X}^\infty g(x) \overline{f(\psi^{-1}(x))}\,\frac{\mu(\psi^{-1}(x))}{\mu(x)}\,\mu(x)\Big|\\
&\Le
\|g\| \sum_{x\in X}^\infty |f(\psi^{-1}(x))|^2 \Big(\frac{\mu(\psi^{-1}(x))}{\mu(x)}\Big)^2\mu(x),\quad g\in\dz{C_\psi}.
\end{align*}
Hence $f\in\dz{C_\psi^*}$. This proves $\dz{C_\psi}\subseteq \dz{C_\psi^*}$. The reverse inclusion can be proved in a similar fashion. Namely, by \eqref{21.01.11.02}, $\sum_{x\in X_+} |f(\psi^{-1}(x))|^2 \frac{\mu(\psi^{-1}(x))^2}{\mu(x)}<\infty$ for every $f\in\dz{C_\psi^*}$. Therefore, by \eqref{21.01.11.03}, we have
    \begin{align*}\allowdisplaybreaks
    \sum_{x\in X_+}^\infty |f(\psi(x))|^2\mu(x)
    &=
    \sum_{x\in X_+}^\infty |f(\psi^{-1}(x))|^2\mu(\psi^{-2}(x))\\
    &=
    \sum_{x\in X_+}^\infty|f(\psi^{-1}(x))|^2\frac{\mu(\psi^{-2}(x))}{\mu(\psi^{-1}(x))}\,\mu(\psi^{-1}(x))\\
    &=
    \sum_{x\in X_+}^\infty |f(\psi^{-1}(x))|^2\frac{\mu(\psi^{-1}(x))}{\mu(x)}\,\mu(\psi^{-1}(x))\\
    &=
    \sum_{x\in X_+}^\infty\Big|f(\psi^{-1}(x))\frac{\mu(\psi^{-1}(x))}{\mu(x)}\Big|^2\mu(x)<\infty,
    \end{align*}
for every $f\in\dz{C_\psi^*}$. This and \eqref{17.01.11.01} imply $\dz{C_\psi}\subseteq \dz{C_\psi^*}$. Moreover, by comparing left and right hand sides we get that $\|C_\psi^*f\|=\|C_\psi f\|$ for every $f\in\dz{C_\psi^*}=\dz{C_\psi}$. Hence $C_\psi$ is normal. Since $\psi\in[\phi]$, we see that $C_\phi$ is normal.
\end{proof}
\begin{rem}
Another way of proving implication (ii) $\Rightarrow$ (i) of Theorem \ref{21.01.11.01} is to show that $C|_{\ee}$ is essentially selfadjoint and use \cite[Theorem 5]{sto-sza4}.
However, though less general, our direct approach is more elementary and fits better to the context.
\end{rem}
Lemma \ref{01.08.06.09} enables us to prove that $C_\phi$ is unitarily equivalent to
orthogonal sum of two-sided backward shift operators whenever condition \eqref{01.04.06.09} is satisfied. Before we state the result formally, let us introduce some notation. If $\gamma=\{\gamma_k\}_{k=-\infty}^\infty$ is a sequence of non-negative real numbers, then $\gamma(k)=\gamma_k$, $k\in\zbb$,
defines a measure\footnote{\; For simplicity, both the sequence and the measure are denoted by the same symbol.} $\gamma$ on $2^\zbb$.
The two-sided backward shift $S$ in $\ell^2(\gamma)$ is the operator given by the formula
    \begin{align*}
    \dz{S}&=\Big\{f\in \ell^2(\gamma)\colon \sum_{n\in \zbb}^\infty |f(n+1)|^2\gamma_n<\infty\Big\},\\
    S\chi_{\{k+1\}}&=\chi_{\{k\}}, \quad k\in\zbb.
    \end{align*}
If $\gamma=\{\gamma_k\}_{k=1}^n$, $n\in\N$, is a finite sequence of non-negative real numbers, then, in a same fashion, $\gamma(k)=\gamma_k$ defines a measure $\gamma$ on $2^{\{1,\ldots,n\}}$. By the two-sided backward shift on $\ell^2(\gamma)$ we understand the operator given by
    \begin{align*}
    S\chi_{\{i+1\}}=\left\{%
\begin{array}{ll}
    \chi_{\{i\}}, & \hbox{for $i=1,\ldots,n-1$;} \\
    \chi_{\{n\}}, & \hbox{for $i=0$.} \\
\end{array}%
\right.    .
    \end{align*}
\begin{pro}\label{24.01.11.01}
$C_\phi$ is normal if and only if there exists a subset $\mathcal{J}$ of $X_+$ and a family $\{\gamma(x)\colon x\in \mathcal{J}\}$ of geometric sequences
{\em(}finite\footnote{\;\label{fingeo} A finite sequence is geometric if it is a constant one.} or infinite{\em)} of positive real numbers
such that $C_\phi$ is unitarily eqivalent to $\bigoplus_{x\in \mathcal{J}}^\infty S(x)$,
where $S(x)$ is the two-sided backward shift in $\ell^2(\gamma(x))$.
\end{pro}
\begin{proof}
If $C_\phi$ is normal, then by Lemma \ref{01.08.06.09} we may assume that transformation $\phi$ is bijective and $\phi(x)=x$ for all $x\in X\setminus X_{+}$. Let $\mathcal{J}$ be a subset of $X_+$ such that $X_+=\bigcup_{x\in \mathcal{J}}\oo(x)$ and $\oo(x)\cap\oo(x')=\varnothing$ for all $x,x'\in \mathcal{J}$, $x\neq x'$. Clearly, $\ell^2(\mu)=\bigoplus_{x\in \mathcal{J}}\chi_{\oo(x)}\ell^2(\mu)$. Moreover, $\chi_{\oo(x)}\ell^2(\mu)$ is reducing for $C_\phi$ for every $x\in \mathcal{J}$. These two facts imply that $C_\phi=\bigoplus_{x\in \mathcal{J}}C_\phi|_{\chi_{\oo(x)}\ell^2(\mu)}$. Therefore,
it suffices to show that for every $x\in \mathcal{J}$, $C_\phi|_{\chi_{\oo(x)}\ell^2(\mu)}$ is unitarily equivalent to the shift operator $S(x)$ acting in $\ell^2(\gamma(x))$. This is evident if $\oo(x)=\{x\}$. Hence, we fix $x\in \mathcal{J}$ such that $\oo(x)\neq\{x\}$.
If $\oo(x)$ is finite, then $\oo(x)=\{x,\phi(x),\ldots,\phi^{m_x}(x)\}$ for some $m_x\in\N$. We may assume that $\oo(x)\neq\{x,\phi(x),\ldots,\phi^{k}(x)\}$ for any $k<m_x$. Define $\gamma(x)=\big(\mu(x), \mu(\{\phi(x)\}),\ldots, \mu(\{\phi^{m_x}(x)\})\big)$.
In case $\oo(x)$ is infinite we have $\oo(x)=\big\{\phi^k(x)\big\}_{k=-\infty}^\infty$. Define $\gamma(x)=\{\mu(\phi^{k}(x))\}_{k=-\infty}^\infty$.
In both the situations, by Lemma \ref{01.08.06.09} and Proposition \ref{orbity}, $\gamma(x)$ is a geometric sequence. As easily seen, the operator $U(x)\colon \chi_{\oo(x)}\ell^2(\mu)\to\ell^2(\gamma(x))$ determined by $U(x)\chi_{\{\phi^k(x)\}}=\chi_{\{k+1\}}$ is unitary and $U_xC_\phi|_{\chi_{\oo(x)}\ell^2(\mu)}=S(x)U_x$. This proves the ''only if'' part of  the claim. The ''if'' part is obvious.
\end{proof}
By the above proposition, any normal composition operator can be regarded as a orthogonal sum of shift operators. This leads to a natural question concerning subnormal composition operators in $\ell^2$-spaces. Namely, we can ask whether every subnormal composition operator have normal extension in form of composition operator acting in $L^2$-space over a discrete measure space (see the paper \cite{jab-jun-sto-next} which deals with this kind of problem within the class of weighted shifts on directed trees).
\begin{opq}
Let $C_\phi$ be a non-normal subnormal composition operator acting in an $\ell^2$-space. Does there exists a composition operator $C_\varPhi$ acting in an $\ell^2$-space, which is a normal extension of $C_\phi$?
\end{opq}
As a corollary to Proposition \ref{24.01.11.01} we also get information when a transformation of a countable set induces unbounded normal composition operator
and which spaces $\ell^2$-spaces admit unbounded normal composition operator.
\begin{cor}
Let $Z$ be a countably infinite set. Then the following assertions are valid$:$
    \begin{enumerate}
    \item Suppose that $\tau\colon Z\to Z$ is bijective and poses infinite number of distinct and infinite orbits. Then there is a measure $\varrho$ on $2^Z$ such that $\varrho(z)>0$ for every $z\in Z$ and $C_\tau$ is unbounded normal operator in $\ell^2(\varrho)$.
    \item Suppose that $\varrho$ is a measure on $2^Z$ such that $\{ \varrho(z)\colon z\in Z\}$ can be decomposed into infinite number of geometric sequences with common ratios tending either to $+\infty$ or $0$.
    Then there is a mapping $\tau\colon Z\to Z$ such that $C_\tau$ is unbounded normal operator in $\ell^2(\varrho)$.
    \end{enumerate}
\end{cor}
Now, we turn our attention to formally normal composition operators in $\ell^2$-spaces.
As it was proven in \cite[Theorem 9.4]{bud-jab-jun-sto-Comp}, within the class of unbounded composition operators in $L^2$-spaces there is no difference
between normal and formally normal operators. We show below that, in the context of $\ell^2$-spaces, normality of $C_\phi$ is equivalent to formal normality of $C_\phi|_{\lin\ee}$.
\begin{pro}\label{24.10.09.01}
$C_\phi$ is normal if and only if $\ee\subseteq\dz{C_\phi}$ and $C_\phi|_{\lin\ee}$ is formally normal.
\end{pro}
\begin{proof}
If $C_\phi$ is normal, then $\ee\subseteq\dz{C_\phi}$ by Lemma \ref{01.21.04.09}. Hence, $C_\phi|_{\lin\ee}$ is formally normal.
For the proof of the reverse implication, assume that $\ee\subseteq\dz{C_\phi}$ and $C_\phi|_{\lin\ee}$ is formally normal.
Lemma \ref{01.21.04.09} implies that $\lin\ee\subseteq \dz{C_\phi^* C_\phi}\cap\dz{C_\phi C_\phi^*}$.
On the other hand, formal normality of $C_\phi|_{\lin\ee}$ yields $\|C_\phi f\|=\|C_\phi^* f\|$ for every $f\in\lin\ee$.
Hence, using polarization formula, we obtain $\is{C_\phi^* C_\phi f}{g}=\is{C_\phi C_\phi^* f}{g}$ for all $f,g\in\lin\ee$.
Since $\lin\ee$ is dense in $\ell^2(\mu)$ we see that \eqref{01.04.06.09} is satisfied. Applying Theorem \ref{21.01.11.01} completes the proof.
\end{proof}
Since symmetric operators are formally normal, by Propositions \ref{24.10.09.01} and \ref{24.01.11.01} we get the following (this result was recently generalized in \cite[Proposition B.1]{bud-jab-jun-sto-Subn}).
\begin{cor}
If $C_\phi$ is symmetric, then $C_\phi$ is selfadjoint, bounded and $C_\phi^2=I$.
\end{cor}
Concluding the paper, we gather few results concerning unbounded quasinormal composition operators. The class was characterized in \cite[Proposition 8.1]{bud-jab-jun-sto-Comp}.
We fill up this characterization by proving that in $\ell^2$-spaces quasinormal composition operators are exactly those densely defined composition operators
which satisfy the condition for quasinormality of bounded operators (that this condition characterizes general quasinormal operators in Hilbert spaces was recently proved in \cite[Theorem 3.1]{jab-jun-sto-quasin}).
\begin{thm}\label{17.10.09.01}
The following conditions are equivalent$:$
\begin{enumerate}
\item[(i)] $C_\phi$ is quasinormal,
\item[(ii)] condition \eqref{star} is satisfied and
\begin{align}\label{quasiwarunek}
\frac{\miu{x}}{\mu(x)}=\frac{\miu{y}}{\mu(y)}, \quad y\in\phi^{-1}(\{x\}), x\in X_+,
\end{align}
\item[(iii)]    $C_\phi$ is densely defined and  $C_\phi^* C_\phi C_\phi=C_\phi C_\phi^* C_\phi$,
\item[(iv)]     $\ee\subseteq\dz{C_\phi^* C_\phi C_\phi}\cap\dz{C_\phi C_\phi^* C_\phi}$ and $C_\phi^* C_\phi C_\phi|_{\lin\ee}=C_\phi C_\phi^* C_\phi|_{\lin\ee}$.
\end{enumerate}
\end{thm}
\begin{proof}
(i)$\Leftrightarrow$(ii) This follows from \cite[Proposition 8.1]{bud-jab-jun-sto-Comp} by \eqref{rnatomic}.

(iii)$\Rightarrow$(iv) Evident, since by Lemma \ref{01.21.04.09} we have $\ee\subseteq\dz{C_\phi C_\phi^* C_\phi}$.

(iv)$\Rightarrow$(ii) Clearly, by Lemma \ref{01.21.04.09}, for every $x,y\in X_+$ we have
    \begin{align}\label{01.11.09.01}
        \big(C_\phi^*C_\phi C_\phi \chi_{\{x\}}\big)(y)&=\frac{\mu\circ\phi^{-1}(y)}{\mu(y)} \chi_{\phi^{-1}(x)}(y)\\
        \label{01.11.09.02}
        \Big(C_\phi C_\phi^* C_\phi \chi_{\{x\}}\Big)(y)&=\tfrac{\miu{x}}{\mu(x)}\, \chi_{\phi^{-1}(\{x\})}(y)
    \end{align}
By comparing right-hand sides in the display above we obtain (ii).

(ii)$\Rightarrow$(iv) By Lemma \ref{01.21.04.09}, we know that $\ee\subseteq\dz{C_\phi C_\phi^* C_\phi}$.
Now we show that $\ee\subseteq\dz{C_\phi^* C_\phi C_\phi}$.
Indeed, let $x\in X_+$. Then we have
\begin{align*}
    \frac{\big(\miu{x}\big)^{2}}{\mu(x)}
    &=
    \frac{\miu{x}}{\mu(x)} \sum_{y\in\phi^{-1}(\{x\})} \mu(y)\\
    &=
    \sum_{y\in\phi^{-1}(\{x\})}\frac{\miu{y}}{\mu(y)}\, \mu(y)\\
    &=
    \sum_{y\in\phi^{-1}(\{x\})} \miu{y}
    =
    \mu\circ\phi^{-2}(x),
    \end{align*}
which, together with $\ee\in\dz{C_\phi}$, \eqref{rnatomic} and \cite[Display (3.6)]{bud-jab-jun-sto-Comp}, yields $\chi_{\{x\}}\in\dz{C_\phi^2}$.
Moreover, by the Cauchy-Schwarz inequality and \eqref{quasiwarunek}, for any $f\in\dz{C_\phi}$ we have
    \begin{align*}
    |\is{C_\phi f}{\chi_{\phi^{-2}(\{x\})}}|
    &\Le
    \sum_{y\in\phi^{-2}(\{x\})} |f(\phi(y))|\, \mu(y)
    =
    \sum_{z\in\phi^{-1}(\{x\})} |f(z)|\, \miu{z}\\
    &=
    \sum_{z\in\phi^{-1}(\{x\})} |f(z)|\, \frac{\miu{z}}{\mu(z)}\, \mu(z)\\
    &\Le
    \|f\|\,\bigg(\sum_{z\in\phi^{-1}(\{x\})} \bigg(\frac{\miu{z}}{\mu(z)}\bigg)^2\,\mu(z)\bigg)^{1/2}\\
    &=
    \|f\|\,\bigg(\frac{\miu{x}}{\mu(x)}\bigg)^2 \big(\miu{x}\big)^{1/2}.
    \end{align*}
This proves that $C_\phi^2\chi_{\{x\}}\in\dz{C_\phi^*}$. Since $x\in X_+$ can be arbitrary chosen, we obtain that $\ee\in\dz{C_\phi^* C_\phi C_\phi}$.
Now, using \eqref{01.11.09.01}, \eqref{01.11.09.02} and \eqref{quasiwarunek} we see that $C_\phi^*C_\phi C_\phi|_{\lin\ee}$ and $C_\phi C_\phi^*C_\phi|_{\lin\ee}$ coincide.

(ii)$\Rightarrow$(iii) $\overline{\dz{C_\phi}}=\ell^2(\mu)$ by Lemma \ref{01.21.04.09}. By \cite[Proposition 8.1\,(i)]{bud-jab-jun-sto-Comp} we know that
    \begin{align*}
    \dz{C_\phi^*C_\phi}=\Big\{f\in\ell^2(\mu)\colon\sum_{x\in X_+} \tfrac{|f(x)|^2 \big(\miu{x}\big)^2}{\mu(x)}<\infty\Big\}.
    \end{align*}
This, \eqref{17.01.11.01} and \eqref{quasiwarunek} imply that
    \begin{align*}
    \dz{C_\phi C_\phi^*C_\phi}&=\Big\{f\in\ell^2(\mu)\colon\sum_{x\in X_+}  \tfrac{|f(x)|^2 \big(\miu{x}\big)^3}{(\mu(x))^2}<\infty\Big\},\\
    \dz{C_\phi^*C_\phi C_\phi}&=\Big\{f\in\ell^2(\mu)\colon\sum_{x\in X_+}  \tfrac{|f(\phi(x))|^2 \big(\miu{x}\big)^2}{\mu(x)}<\infty\Big\}.
    \end{align*}
Therefore, applying \eqref{quasiwarunek} again, we get that
    \begin{align*}
    \dz{C_\phi^*C_\phi C_\phi}
    &=
    \Big\{f\in\ell^2(\mu)\colon\sum_{x\in X}^\infty |f(\phi(x))|^2\tfrac{\big(\miu{\phi(x)}\big)^2}{\big(\mu(\phi(x)\big)^2}\,\mu(x)<\infty\Big\}\\
    &=
    \Big\{f\in\ell^2(\mu)\colon\sum_{x\in X}^\infty |f(x)|^2\tfrac{\big(\miu{x}\big)^2}{\big(\mu(x)\big)^2}\tfrac{\miu{x}}{\mu(x)}\,\mu(x)<\infty\Big\}\\
    &=
    \dz{C_\phi C_\phi^*C_\phi}.
    \end{align*}
By Lemma \ref{01.21.04.09} we know that $C_\phi^*(\lin\ee)\subseteq\lin\ee$ and thus, again by Lemma \ref{01.21.04.09}, $\lin\ee\subseteq \dz{C_\phi^*C_\phi C_\phi^*}$. Since $C_\phi^*C_\phi$ is selfadjoint, we have $C_\phi^*C_\phi C_\phi^*\subseteq\big(C_\phi C_\phi^* C_\phi\big)^*$. Von Neumann theorem (cf. \cite[Theorem 3.3.7]{bir-sol}), we see that $C_\phi C_\phi^* C_\phi$ is closable.
This and implication (ii)$\Rightarrow$(iv) yields the following inclusions
    \begin{align}\label{inkluzje}
    \overline{C_\phi^*C_\phi C_\phi|_{\lin\ee}}
    =
    \overline{C_\phi C_\phi^* C_\phi|_{\lin\ee}}
    \subseteq
    \overline{C_\phi C_\phi^* C_\phi}.
    \end{align}
Now, we will prove that $C_\phi^*C_\phi C_\phi\subseteq\overline{C_\phi^*C_\phi C_\phi|_{\lin\ee}}$.
Indeed, let $f\in\dz{C_\phi^*C_\phi C_\phi}$. Then, by \eqref{01.11.09.01} we have
    \begin{align*}
    \big(C_\phi^*C_\phi C_\phi f\big)(x)=\frac{\miu{x}}{\mu(x)}\,f(\phi(x)), \quad x\in X.
    \end{align*}
Fix $\varepsilon>0$. Then there is a finite subset $Y$ of $X$ such that
    \begin{align}\label{final}
    \sum_{x\in X\setminus Y}^\infty \big|\frac{\miu{x}}{\mu(x)}\,f(\phi(n))\big|^2 \mu(x)\Le\varepsilon,
    \quad
    \sum_{x\in X\setminus Y}^\infty |f(x)|^2 \mu(x)\Le\varepsilon
    \end{align}
Set $Z=Y\cup\phi(Y)$ and $\tilde f=\chi_{Z}f$. Then $\tilde f\in\lin\ee$. Moreover, by \eqref{final}, we have
    \begin{align*}
    \|f-\tilde f\|\Le\varepsilon,
    \quad
    \|C_\phi^*C_\phi C_\phi \tilde f-C_\phi^*C_\phi C_\phi f\|\Le \varepsilon
    \end{align*}
Since $\varepsilon$ can be arbitrary, this yields $C_\phi^*C_\phi C_\phi\subseteq\overline{C_\phi^*C_\phi C_\phi|_{\lin\ee}}$. As a consequence, by \eqref{inkluzje}, we get the inclusion
    \begin{align*}
    C_\phi^*C_\phi C_\phi
    \subseteq
    \overline{C_\phi C_\phi^* C_\phi}.
    \end{align*}
This together with $\dz{C_\phi^*C_\phi C_\phi}=\dz{C_\phi C_\phi^*C_\phi}$ imply that $C_\phi C_\phi^* C_\phi=C_\phi^*C_\phi C_\phi$. Hence we get (iii) and the proof is complete.
\end{proof}
\begin{rem}
Recently, new characterization of quasinormality of composition operators in $L^2$-spaces was invented (see \cite[Theorem 3.1]{bud-jab-jun-sto-multiplicative}). The main ingredient of it is the equivalence between quasinormality of $C_\phi$ and condition $\mathsf{h}_\phi^n=\mathsf{h}_{\phi^n}$ a.e. $[\mu]$ for every $n\in\mathbb{N}$. In the context of composition operators in $\ell^2$-spaces the characterization reads as follows: $C_\phi$ is quasinormal if and only if $\big(\frac{\mu\circ \phi^{-1}(x)}{\mu(x)}\big)^n=\frac{\mu\circ \phi^{-n}(x)}{\mu(x)}$ for every $x\in X_+$ and every $n\in\N$. The proof
\end{rem}
\begin{rem}
It is worth pointing out that for quasinormal $C_\phi$ its symbol $\phi$ must be almost surjective, i.e., $\mu(X\setminus\phi(X))=0$ (use \eqref{quasiwarunek}), which can be easily deduced from Theorem \ref{17.10.09.01}. This in turn implies, by Lemma \ref{01.21.04.09}, existence of surjective representative of $[\phi]$. As a consequence a quasinormal composition operator $C_\phi$ must be automatically normal whenever a restriction of $\phi$ to $X_+$ is injective. These results turn out to be valid for hyponormal operators (the class of quasinormal operators is included in that of hyponormal operators), which follows from \cite[Proposition 6.2 and Corollary 6.3]{bud-jab-jun-sto-Comp}.
\end{rem}

\end{document}